\documentclass[12pt, a4paper]{amsart}
\usepackage{amssymb,amsmath,enumitem}
\usepackage[foot]{amsaddr}
\usepackage{graphicx}
\usepackage{comment}
\usepackage[dvipsnames]{xcolor}
\usepackage[linktocpage]{hyperref}
\hypersetup{colorlinks=true, citecolor=JungleGreen, linkcolor=Plum, urlcolor=Sepia}
\newtheorem{theorem}{Theorem}[section]
\newtheorem{lemma}[theorem]{Lemma}

\theoremstyle{definition}
\newtheorem{definition}[theorem]{Definition}

\theoremstyle{remark}
\newtheorem{remark}[theorem]{Remark}

\numberwithin{equation}{section}
\newcommand{\R}{\mathbb R} 
\newcommand{\N}{\mathbb N} 
\newcommand{\dom}{\Omega} 

\newcommand{\bdy}{\partial^\infty\dom} 
\newcommand{\cl}{\overline{\dom}^\infty} 
\newcommand{\GP}{{\bf{G}}} 
\newcommand{\CF}{C} 
\newcommand{\RM}{\mathcal{M}^{+}(\dom)} 
\newcommand{\K}{{\bf{K}}} 

\newcommand{\qtext}{\quad\text}
\newcommand{\con}{c} 

\newcommand{\A}{\mathcal{A}}
\newcommand{\LL}{\mathcal{L}}
\newcommand{\LGP}{\GP_{\LL}}
\newcommand{\I}{\mathbf{I}}
\newcommand{\dlim}{\displaystyle\lim}

\makeatletter
\newcounter{con}
\def\cont{\refstepcounter{con}\con_{\thecon}}
\newcommand{\conlabel}[1]{\cont\label{#1}}
\newcommand{\conref}[1]{\con_{\scriptsize\ref{#1}}}

\@namedef{subjclassname@2020}{%
  \textup{2020} Mathematics Subject Classification}
\makeatother
\title[Dirichlet problem for Lane--Emden type equations]
{Dirichlet problem for Lane--Emden type equations\\
with several sublinear terms}
\author{Toe Toe Shwe}
\address[Toe Toe Shwe]{Sirindhorn International Institute of Technology, Thammasat University, Pathum Thani 12120, Thailand}
\email{\href{mailto:d6722300172@g.siit.tu.ac.th}{d6722300172@g.siit.tu.ac.th (T.T. Shwe)}}
\thanks{}

\author{Kentaro Hirata}
\address[Kentaro Hirata]{Department of Mathematics, Graduate School of Advanced Science and Engineering,
Hiroshima University, Higashi-Hiroshima 739-8526, Japan}
\email{\href{mailto:hiratake@hiroshima-u.ac.jp}{hiratake@hiroshima-u.ac.jp (K. Hirata)}}
\thanks{}

\author{Adisak Seesanea}
\address[{Corresponding Author}: Adisak Seesanea]{Sirindhorn International Institute of Technology, Thammasat University, Pathum Thani 12120, Thailand}
\email{\href{adisak.see@siit.tu.ac.th}{adisak.see@siit.tu.ac.th (A. Seesanea)}}
\thanks{}
\begin{document}
\begin{abstract}
We prove the existence, uniqueness, and sharp bilateral pointwise estimates for positive bounded solutions 
to the Lane--Emden type problem
\[
\begin{cases}
 	\mathcal{L} u = \sum\limits_{i=1}^{m}\sigma_{i} u^{q_{i}}+\sigma_0,
	\quad u\geq0 & \text{in } \Omega, \\
	\liminf \limits_{x \rightarrow y} u(x) = f(y),  & y \in \partial^\infty\Omega,
\end{cases}
\]
where $0 < q_{i} < 1$.
Here $\mathcal{L} u = - \operatorname{div}(\A \nabla u)$
is a uniformly elliptic operator with bounded coefficients,
$\sigma_{i}$ is a nonnegative locally finite Borel measure on
an  $\A$-regular domain $\Omega \subset \R^n$
which possesses a positive Green function associated with $\mathcal{L}$,
and $f$ is a nonnegative continuous function on the boundary $\partial^\infty\Omega$.
An analogous result for positive continuous solutions to the problem
is also illustrated.  

Our method can be adapted to address related sublinear problems
with zero boundary conditions
involving the fractional Laplace operator $(-\Delta)^{\alpha}$
for $0< \alpha < n/2$, in place of $\mathcal{L}$,
in $\mathbb{R}^n$ as well.
\end{abstract}
\maketitle
\tableofcontents
\section{Introduction}\label{sect:intro}
\subsection{Overview}
 We study the Dirichlet problem for
 the Lane--Emden type equation
\begin{equation}\label{main:eq}
\begin{cases}
	\LL u = \sum\limits_{i=1}^{m}\sigma_{i} u^{q_{i}}+\sigma_0,
	\quad u\geq0 &\text{in } \dom, \\
	\liminf \limits_{x \to y} u(x) = f(y),  &  y \in \bdy,
\end{cases}
\end{equation}
involving several sublinear growth terms where
$0 < q_i <1$  ($i = 1, \dots, m$),
and the coefficients $\sigma_i$ ($i = 0, 1,\dots, m$)
are nonnegative locally integrable functions,
or more generally
nonnegative locally finite Borel measures on
a domain $\dom\subset\R^n$ $(n\ge2)$.
The set of all such measures will be denoted by $\RM$.

The operator
$\LL u = - \operatorname{div}(\A \nabla u)$
with $\A: \dom \rightarrow \R^{n \times n}$
being a real $n \times n$ symmetric matrix-valued measurable function on $\dom$
is assumed to satisfy
the uniform ellipticity condition, that is,
there exists a positive constant $\beta \ge 1$ such that 
\begin{equation}\label{cond:ellipticity}
\beta^{-1} |\xi|^{2} \leq \A(x)\xi \cdot \xi \leq \beta |\xi|^{2},
\end{equation}
for almost every $x \in \dom$ and for every $\xi \in \R^n$.
When $\A(x) = I_{n}$ is the $n \times n$ identity matrix,
$\LL = -\Delta$ the classical Laplace operator on $\R^n$.

The domain $\dom$
is assumed to be $\A$-regular (or simply, regular)
for the Dirichlet problem (cf. \cite{HKM})
and possess a positive Green function $G_{\LL}$
associated with $\LL$
(see e.g. \cite{Ken, LSW}).
Also, $f \in \CF^{+}(\bdy)$;
the set of all nonnegative continuous functions on the boundary of $\dom$,
denoted by $\bdy$,
in the one point compactification $\R^{n} \cup \{ \infty \}$.
Therefore,
$\infty$ belongs to $\bdy$ if and only if $\dom$ is unbounded.

We will understand a {\em solution} $u$ to \eqref{main:eq}
as a nonnegative function on $\dom$
belonging to $L_{\rm loc}^{q_i}(\dom, \sigma_i)$
for all $i=1, \dots, m$,
which satisfies the integral equation
\begin{equation}\label{inteq1}
	u = \sum_{i=1}^{m} \LGP (u^{q_i} d\sigma_{i}) + \LGP\sigma_0 + H_f \quad \text{in } \dom,
\end{equation}
where $H_f$ is the Perron solution (cf. \cite{HKM})
of the Dirichlet problem  
\[
\begin{cases}
  \LL h=0 & \text{in } \dom, \\
 \lim \limits_{x \rightarrow y} h(x) = f(y), &  y \in \bdy.
\end{cases}
\]
Here the potential $\GP \sigma$ of $\sigma \in \RM$
associated with a positive lower semicontinuous function
$G : \dom \times \dom \to (0, +\infty]$
is defined by
\begin{equation} \label{genpot}
	\GP \sigma (x) = \int_\dom G (x,y) \, d\sigma (y),
	\quad x \in \dom.
\end{equation}
Note that
$\GP \sigma $ is again lower semicontinuous on $\Omega$
(see \cite{Bre}).
We call such $G$ a {\em kernel}.
When $G= G_{\LL}$,
we use $\LGP\sigma$ to denote
the corresponding Green potential of $\sigma$.

In addition, we will also investigate the following sublinear nonlocal problem 
\begin{equation}\label{main:frac}
\begin{cases}
 	(-\Delta)^{\alpha} u
	= \sum\limits_{i=1}^{m}\sigma_{i} u^{q_{i}} + \sigma_0,
	\quad u\geq0  & \text{in } \R^n, \\
	\liminf \limits_{x \rightarrow \infty} u(x) = 0,  
\end{cases}
\end{equation}
where $q_i$ and $\sigma_i$ are as above and
$(-\Delta)^{\alpha}$, with  $0<\alpha<n/2$,
is the fractional Laplace operator on $\R^n$.
In the same spirit as above, \eqref{main:frac} will be understood in the integral sense 
\begin{equation}\label{inteq-q}
	u
  =
  	\sum_{i=1}^{m} \I_{2\alpha}(u^{q_i} d\sigma_{i})
	+
	\I_{2\alpha} \sigma_0 \quad \text{in } \R^{n},
\end{equation}
where we always implicitly assume
$u \in L_{\rm loc}^{q_i}(\R^{n}, \sigma_i)$
 for all $i = 1, \dots, m$ so that each
$u^{q_i} d\sigma_{i} \in \mathcal{M}^{+}(\R^{n})$. 

The Riesz potential $\I_{2\alpha} \sigma$ of
$\sigma \in \mathcal{M}^{+}(\R^n)$
with order $0 < \alpha < n/2$ is given by \cite{AH}
\[
	\I_{2\alpha} \sigma (x) 
  =
  	\int_{\R^n} k_{2 \alpha}(x,y) \, d\sigma(y)
  =
  	\con(n,\alpha) \int_{0}^{\infty}
	\frac{\sigma (B(x,r))}{r^{n-2\alpha}} \frac{dr}{r},
	\quad  x \in \R^n,
\]
where $k_{2 \alpha}(x,y) = |x -y|^{2\alpha - n}$ is the Riesz kernel and
$B(x,r) = \{ y \in \R^n : |x-y| < r \}$
for $x \in \R^n$ and $r >0$.
In particular, when $\alpha=1$, $\I_{2}\sigma$ becomes the classical Newtonian potential. 
Note that $ \liminf \limits_{x\to\infty} \ \I_{2\alpha} \sigma(x) =0$ whenever $ \I_{2\alpha} \sigma \not\equiv +\infty$
(see \cite{Lan, Miz}).
\subsection{Main Goal}
This work aims to develop a method for simultaneously establishing necessary and sufficient conditions for the existence and uniqueness of positive bounded solutions to the Dirichlet problems \eqref{main:eq} and \eqref{main:frac}.
Our framework further yields sharp two-sided pointwise estimates for such solutions.
Moreover, we obtain an analogous result for positive continuous solutions to similar problems of the type \eqref{main:eq}.
\subsection{State of the Art}
One primary motivation for studying
the Lane--Emden type equation \eqref{main:eq} with sublinear terms
stems from its connection to the porous medium equation (PME).
In particular, by setting
$\LL = -\Delta$, $\sigma_0 = 0$, $m = 1$, $q_1=q$
and letting $\sigma_1 = \sigma(x)$
be a nonnegative measurable function, \eqref{main:eq} reduces to
\begin{equation}\label{bk}
	-\Delta u = \sigma(x) u^{q} \quad \text{in } \dom.
\end{equation}
Such equations are intimately linked to the analysis of the associated PME
$\sigma (x)  U_ t = \Delta (U^{Q})$ in $\dom \times (0, \infty)$, 
where $Q = 1/q > 1$ (see \cite{JL}). 

Brezis and Kamin's seminal work \cite{BK} established
the existence and uniqueness of a positive bounded solution $u$
to \eqref{bk} 
with zero boundary values in the case
$\sigma(x) \in L_{\rm loc}^{\infty}(\R^n)$ and
$\dom=\R^{n}$ ($n \geq 3$).
They also proved
the sharp two-sided pointwise estimates
\begin{equation}\label{bkest}
	\con^{-1} (\I_{2}\sigma)^{\frac{1}{1-q}}
  \le u \le  \con \, \I_{2} \sigma  \quad \text{in } \R^n. 
\end{equation}
See also the work of Mabrouk \cite{KE} for a treatment of the existence of a bounded solution to \eqref{bk} in an arbitrary domain $\dom$.

We would like to refer the reader to \cite{ADPU24,SFU24} for recent related results concerning the existence and pointwise behaviors of bounded solutions to a more general class of sublinear problems of the form 
\eqref{bk}.

For the nonlinear counterparts of \eqref{bk}, a characterization of positive $p$-superharmonic (or equivalently locally renormalized) solutions to the corresponding $p$-Laplace equations in the sub-natural growth case $0 < q < p-1$, satisfying the Brezis--Kamin type estimates in terms of the nonlinear Wolff potentials, have been established in the work of Verbitsky \cite{Ver23}
(see also \cite{CV3}).

We seize this opportunity to mention the work \cite{HS} which investigates the existence and bilateral pointwise estimates of positive continuous solutions to \eqref{main:eq} in the case
$\LL = -\Delta$ and $m=1$.
Furthermore, the uniqueness of such solutions is established
under certain integrability conditions of the coefficient $\sigma_{1}$.
For further developments in this direction, see
\cite{BM1, BM2, KH1,  MM, MH}.

In recent years, we have seen significant progress in the study of several classes of solutions to problems of the type \eqref{main:eq} and \eqref{main:frac}. For instance, in \cite{SV2}, the last-named author and Verbitsky established necessary and sufficient conditions for the existence of positive finite generalized energy solutions to \eqref{main:eq} with zero boundary values in the case $m = 1$
(see also \cite{SV1} for a related result).
Subsequently, Verbitsky \cite{VE1} investigated the existence and uniqueness of general positive solutions to an exterior fractional Laplace problem for $m = 1$. More recently, May and the last-named author \cite{MS2} provided sufficient conditions for the existence of positive solutions with zero boundary values to \eqref{main:frac} in Lebesgue spaces when $m = 1$, and later \cite{MS1} extended the results to the case $m \geq 1$ in the more refined framework of Lorentz spaces. 
\subsection{Main Result}

Throughout this paper, we suppose that
\begin{itemize}\itemsep=3pt
\item
$\{ q_i \}_{i=1}^{m}\subset(0,1)$
is nonincreasing, and
\item
$\{ \sigma_i \}_{i=0}^{m} \subset \RM$ satisfies
$\sigma_{i} \not\equiv 0$ for all $i=1,\dots,m$.
\end{itemize}
We say that a positive solution $u$ to \eqref{main:eq}
is {\em minimal}
if $u \leq \tilde{u}$ in $\dom$
for any nontrivial supersolution $\tilde{u}$ to \eqref{main:eq}.

We now state our main result as follows.

\begin{theorem}\label{main:thm}
Let $\dom$ be a regular domain in $\R^n$ $(n\ge2)$
with a positive Green function $G_{\LL}$
and let $f \in \CF^{+}(\bdy)$.
Then there exists a positive bounded solution to \eqref{main:eq}
if and only if  
\begin{equation}\label{cond:greenbdd}
	\text{$\LGP\sigma_i$ is bounded on $\dom$ for all $i = 0,1, \dots, m$}. 
\end{equation}
In this case, there exists a minimal one,
and also all positive bounded solutions $u$ to \eqref{main:eq}
satisfy the two-sided pointwise estimates
\begin{equation}\label{eq:est i}
	\conref{c:lb-gen1}
	\sum_{i=1}^{m} (\LGP\sigma_{i})^{\frac{1}{1-q_{i}}}
	+ \LGP\sigma_0 + H_f
  \leq
	u
  \leq
	\conref{c:ub-gen1}^{q_{1}}
	\sum_{i=1}^{m} \LGP\sigma_{i}
	+ \LGP\sigma_0 + H_f
	\quad \text{in } \dom,
\end{equation}
where $\conlabel{c:lb-gen1} := (1-q_{1})^{\frac{1}{1-q_{1}}}$ and
\[
	\conlabel{c:ub-gen1}
  :=
  	\max\left\{1,
	\left(
	\sum_{i=0}^{m} \| \LGP\sigma_{i} \|_{\infty} + \|f\|_{\infty}
	\right)^{\frac{1}{1-q_{1}}}
	\right\}.
\]
If, in addition $\dom$ is a bounded uniform domain or
a uniform cone (when $\LL = -\Delta$ and $n\geq3$),
then the positive bounded solution is unique.
\end{theorem}

See \cite{HW,KH4}
for the definitions of uniform domains and uniform cones.

\begin{remark}
It is worth pointing out that
the statement of Theorem \ref{main:thm} remains valid
when $f \equiv 0$,
for arbitrary (not necessarily regular) domains $\dom$
with a positive Green function $G_{\LL}$.
As noted above, such results were obtained in \cite{BK}
using a different  method from ours,
in the case of $\dom= \R^n$ $(n \geq 3)$,
$\LL = -\Delta$, $\sigma_{0} = 0$ (homogeneous case),
and a single sublinear term $m=1$ with
nonnegative $\sigma_{1}\in L^{\infty}_{\rm loc}(\R^n)$.
\end{remark}

In our approach to Theorem \ref{main:thm} presented in Section \ref{sect:bddsol},
we investigate positive bounded solutions to
the corresponding integral equation of the form \eqref{inteq1}
in a more general setting
\begin{equation}\label{inteq_gen}
	u = \sum_{i=1}^{m} \GP (u^{q_i} d\sigma_{i}) + \GP \sigma_0 + H, \quad
	u \ge 0
  \quad\text{in } \dom,
\end{equation}
where $0 < q_{i} < 1$, $\sigma_{i} \in \RM$,
$H$ is a nonnegative bounded function on an arbitrary domain $\dom \subset \R^n$,
and the kernel 
$G(x,y)$ is mildly assumed to satisfy
the {\em weak maximum principle} (WMP),
also known as the boundedness principle.
That is, there exists a constant $b \geq 1$ such that
\[
	\GP\sigma \leq b \text{ on } \dom
  \quad
	\text{whenever $\sigma \in \RM$ satisfies $\GP \sigma \leq 1$ on $\operatorname{supp}(\sigma)$}.
\] 
Here $\operatorname{supp}(\sigma)$ denotes the support of $\sigma$.
In the case $b =1$,
the kernel $G$ is said to satisfy the
{\em strong (Frostman's) maximum principle} (SMP). 
We note that the Green function $G_{\LL}$
on $\dom$
associated with a uniformly elliptic operator $\LL$
satisfies the SMP (see \cite{Ken, LSW}).
Also, the WMP is satisfied by radially nonincreasing kernels on $\R^n$
(see \cite[Section 2.6]{AH}).
In particular, the Riesz kernel
$k_{2\alpha}(x,y) = |x-y|^{2\alpha - n}$ with $0 < \alpha < n/2$
satisfies the WMP
with $b = 2^{n - 2\alpha}$
(see \cite[Theorem 1.5]{Lan}).

The robustness of our method in the nonlocal setting yields the following result for \eqref{main:frac}.

\begin{theorem}\label{fractional}
Let $0 < \alpha < n/2$, where $n\ge2$.
Then there exists a unique positive bounded solution $u$
to \eqref{main:frac} if and only if   
\begin{equation}\label{cond:rieszbdd}
  \text{$\I_{2\alpha} \sigma_i$ is bounded on $\R^n$ for all $i= 0,1,\dots,m$}. 
\end{equation}
Moreover, $u$ satisfies the two-sided pointwise estimates
\begin{equation}\label{eq:est r}
	\conref{c:lb-gen3}
	\sum_{i=1}^{m} (\I_{2\alpha}\sigma_{i})^{\frac{1}{1-q_{i}}}
	+\I_{2\alpha}\sigma_0 
  \le
	u
  \le
	\conref{c:ub-gen3}^{q_{1}}
	\sum_{i=1}^{m} \I_{2\alpha}\sigma_{i}+\I_{2\alpha}\sigma_0 
	\quad \text{in } \R^n ,
\end{equation}
where
$\conlabel{c:lb-gen3} := (1-q_{1})^{\frac{1}{1-q_{1}}} 2^{\frac{-q_{1} (n-2\alpha)}{(1-q_1)^{2}}}$
and
\[
	\conlabel{c:ub-gen3}
  :=
  	\max\left\{1,
	\left(
	\sum_{i=0}^{m} \|\I_{2\alpha}\sigma_{i}\|_{\infty}  
	\right)^{\frac{1}{1-q_{1}}}
	\right\}.
\]
\end{theorem}

The Kato-type condition stated below was introduced in \cite{KH1,HS} in the classical setting where $\LL = -\Delta$.
We may naturally extend it as follows.

\begin{definition} \label{defkato}
Let $\dom$ be a domain with a positive Green function $G_{\LL}$.
We say that $\sigma \in \RM$ satisfies the {\em $G_{\LL}$-Kato condition}
if
\begin{equation}
	\lim_{r \to 0^+}
	\left(
	\sup_{x \in \dom}
	\int_{\dom \cap B(x,r)} G_{\LL} (x,y) \, d \sigma(y)
	\right) = 0,
\end{equation}
and 
\begin{equation}\label{Kato2}
	\lim_{r \to 0+}
	\left(
	\sup_{x \in \dom}
	\int_{\dom \setminus B(0,1/r)} G_{\mathcal{L}} (x,y) \, d \sigma(y)
	\right) = 0
   \qtext{(when $\dom$ is unbounded)}.
\end{equation}
\end{definition}

By leveraging characterizations of the $G_{\LL}$-Kato condition (see Section \ref{sec:cont}),
together with Theorem \ref{main:thm},
we obtain the following generalization of the recent finding in \cite{HS}.

\begin{theorem}\label{continuous}
Let $\dom$ be a regular domain in $\R^n$ $(n\ge2)$
with a positive Green function $G_{\LL}$
and let $f \in \CF^{+}(\bdy)$.
Then there exists a (minimal) positive  solution $u \in\CF(\cl)$ to
\begin{equation}\label{main:eq2}
\begin{cases}
 	\LL u = \sum\limits_{i=1}^{m}\sigma_{i} u^{q_{i}}+\sigma_0,
	\quad u\geq0  & \text{in } \dom, \\
	\lim \limits_{x \rightarrow y} u(x) = f(y),  & y \in \bdy,
\end{cases}
\end{equation}
if and only if $\sigma_i$ satisfies
the $G_{\LL}$-Kato condition
for all $i= 0,1,\dots,m$. 
Moreover, $u$ satisfies \eqref{eq:est i}.
If, in addition $\dom$ is a bounded uniform domain or
a uniform cone (when $\LL= -\Delta$ and $n\geq3$),
then the positive continuous solution is unique.
\end{theorem}

\subsection{Outline}
The remainder of this article is structured as follows. In Section \ref{sect:bddsol}, we investigate bounded solutions to the integral equation \eqref{inteq_gen} and provide the proofs of Theorems \ref{main:thm} and \ref{fractional}. Section \ref{sec:cont} is devoted to the proof of Theorem \ref{continuous}, where we present two distinct approaches for constructing a positive continuous solution to \eqref{main:eq2}.

\section{Bounded Solutions to Sublinear Problems}\label{sect:bddsol}
We begin this section by collecting essential ingredients for our analysis.
In this section, unless otherwise stated,
we denote by $\dom$ an arbitrary domain in $\R^n$ ($n \geq 1$). 

The first lemma provides an iterated pointwise estimate for potentials associated with WMP kernels.

\begin{lemma}[\cite{GV}]\label{lem:iterated}
Suppose $G$ is a kernel on $\dom \times \dom$
satisfying the \textup{WMP} with constant $b$. 
Let $\sigma \in \RM$ and $s \geq 1$.
Then
\[
	(\GP\sigma)^{s}
  \leq
  	s b^{s-1} \, \GP \big( (\GP\sigma)^{s-1} d\sigma \big)
  \qtext{in } \dom.
\]
\end{lemma}

The next lemma establishes pointwise lower bounds
for supersolutions of sublinear integral equations involving WMP kernels.

\begin{lemma}[\cite{GV}] \label{lem:lowerbound}
Suppose $G$ is a kernel on $\dom \times \dom$
satisfying the \textup{WMP} with constant $b$.
Let $0<q<1$ and $\sigma \in \RM$. 
If $u \in L^{q}_{\rm loc}(\dom, \sigma)$
is a nontrivial nonnegative function satisfying
\begin{equation}\label{supersoln}
	u \geq \GP (u^{q}d\sigma) \qtext{in } \dom,
\end{equation}
then
\[
	u
  \geq
  	(1-q)^{\frac{1}{1-q}} b^{\frac{-q}{1-q}}
	\left( \GP \sigma \right)^{\frac{1}{1-q}}
  \qtext{in } \dom.
\]
\end{lemma}

As in \cite[(3.6)]{VE1}, the above two lemmas yield the following.

\begin{lemma}[\cite{VE1}]\label{lem:Gu>GP}
Under the same assumption as in Lemma \ref{lem:lowerbound},
we have
\[
	\GP(u^q d\sigma)
  \ge
  	(1-q)^{\frac{1}{1-q}}b^{\frac{-q}{1-q}} (\GP\sigma)^{\frac{1}{1-q}}
  \qtext{in } \dom,
\]
whenever $u \in L^{q}_{\rm loc}(\dom, \sigma)$
is a nontrivial nonnegative function satisfying \eqref{supersoln}.
\end{lemma}

A kernel $G$ on $\dom \times \dom$
is said to be {\em quasi-symmetric} (QS)
if there exists a  constant $a \geq 1$ such that 
\[
	a^{-1} G(y,x) \leq G(x,y) \leq a G(y,x),
  \quad \forall x,y \in \dom.
\]
In the case $a =1$, $G$ is said to be {\em symmetric}. 

We say that a symmetric kernel $G$ on $\dom \times \dom$
is {\em quasi-metric} (QM) if
$d(x,y) := 1/G(x,y)$ satisfies the quasi-triangle inequality
\[
	d(x,y) \leq h \, [d(x,z) + d(z,y)],
  \quad \forall x,y,z \in \dom,
\]
with quasi-metric constant $h$.
Without loss of generality,
we may assume that $d$ is nontrivial, i.e.,
$d(x,y) \neq 0$ for some $x, y \in \dom$,
and thus  $h \geq 1/2$.
It was shown in \cite[Lemma 2.1]{VE1} (see also \cite[Lemma 3.5]{QV})
that a QM kernel with quasi-metric constant $h$
satisfies the WMP with constant $b=2h$.
On $\dom = \R^n$,
the Riesz kernel of order $0<\alpha<n/2$
is QM (see \cite{VE1}).


Given a kernel $G$ on $\dom \times \dom$,
we define
\[
	B_{G}(x, r)
  =
  	\left\{ y \in \dom: G(x,y) > \frac{1}{r} \right\}
  \quad \text{for } x \in \dom \text{ and }  r > 0.
\]
Let $B:=B_G(x,r)$, $\sigma \in \RM$ and $0 < q < 1$.
By $\varkappa(B)$,
we denote the least constant $\con$ of
the following localized version of the weighted norm 
inequalities of $(1,q)$-type \cite{QV}: 
\begin{equation}\label{endpointinq}
 	\| \GP \nu \|_{L^{q}(\dom,\,\sigma_B)}
  \leq
  	\con \| \nu \|, \quad \forall \nu \in \RM,
\end{equation}
where $\sigma_{B}$ is the restriction of $\sigma$ to $B$.  
Then the intrinsic nonlinear potential $\K \sigma$,
introduced in \cite{VE1}, is defined by
\[
	\K \sigma (x)
  =
  	\int_{0}^{\infty} \frac{[\varkappa(B_{G}(x,r))]^{\frac{q}{1-q}}}{r^2} \, dr,
  \quad x \in \dom.
\]

The next lemma provides an upper pointwise estimate for potentials associated with QM kernels.

\begin{lemma}[\cite{VE1}] \label{lemupper}
Suppose $G$ is a \textup{QM} kernel on $\dom \times \dom$
with quasi-metric constant $h$.
Let $\nu, \sigma \in \RM$ and $0<q<1$.
Then
\[
	\GP\big( (\GP\nu)^q d\sigma \big)
  \leq
  	(2h)^q (\GP \nu)^q \left[ \GP\sigma+(\K \sigma)^{1-q} \right]
  \quad \text{in } \dom.
\]
\end{lemma}

When the WMP kernel $G$ in Lemma \ref{lem:lowerbound}
is symmetric (or more generally, QS),
one can refine the lower estimates for solutions to \eqref{supersoln}
in terms of the intrinsic potential $\K\sigma$.

\begin{lemma}[{\cite[Lemma 3.2 and (3.6)]{VE1}}] \label{symmetricWMP}
Suppose
$G$ is a symmetric kernel on $\dom \times \dom$
satisfying the \textup{WMP} with constant $b$.
Let $\sigma \in \RM$ and $0<q<1$.
If $u \in L^{q}_{\rm loc}(\dom, \sigma)$
is a nontrivial nonnegative function 
satisfying \eqref{supersoln}, then
\begin{equation}\label{lowerK}
	u
  \ge
  	\GP(u^qd\sigma)
  \ge
  	(1-q)^{\frac{1}{1-q}} b^{\frac{-q}{1-q}} \K \sigma
  \qtext{in } \dom.
\end{equation}
\end{lemma}

\begin{remark}\label{rem_symmetrized}
In Lemma \ref{symmetricWMP},
if the kernel $G$ is only QS
with quasi-symmetry constant $a$,
then \eqref{lowerK} remains valid
with a constant $\con=\con(q,b,a)$ in place of
$(1-q)^{\frac{1}{1-q}}b^{\frac{-q}{1-q}}$.
Indeed, the symmetrized kernel $G^s (x,y) := G(x,y) + G(y,x)$
is symmetric and satisfies
\[
	(1+ a^{-1} ) G(x,y)
  \leq
  	G^s (x,y)
  \leq
  	(1+a) G(x,y),
  \quad \forall x,y \in \dom.
\]
\end{remark}

Our key theorems in this section discuss the existence, uniqueness, and pointwise behavior of bounded solutions to \eqref{inteq_gen} under various settings of the corresponding kernel $G$.

\begin{theorem}\label{bounded}
Let $G$ be a kernel on $\dom \times \dom$ and
let $H$ be a nonnegative bounded function on $\dom$.
\begin{enumerate}[label=\textup{(\roman*)}] \itemsep=0.4em \setlength{\itemindent}{-0em}
\item\label{item:2.6-i}
Suppose $G$ satisfies the \textup{WMP} with constant $b$.
Then the following hold: \medskip
\begin{enumerate}[label=\textup{(\alph*)}] \itemsep=0.4em \setlength{\itemindent}{-0em}
\item\label{item:2.6-a}
If $\GP \sigma_i$ is bounded on $\dom$ for all $i=0,1,\dots , m$,
then there exists a minimal positive bounded solution to \eqref{inteq_gen}.
\item\label{item:2.6-b}
If \eqref{inteq_gen} admits a nontrivial bounded supersolution,
then $\GP \sigma_i$ is bounded on $\dom$ for all $i=0,1,\dots,m$.
\item\label{item:2.6-c}
All nontrivial bounded solutions $u$ to \eqref{inteq_gen}
enjoy the Brezis--Kamin type estimates
 \begin{equation}\label{eq:estimate}
	\conref{c:lb-gen2}
	\sum_{i=1}^{m} (\GP\sigma_{i})^{\frac{1}{1-q_{i}}}
	+\GP\sigma_0 + H
  \le
	u
  \le
	\conref{c:ub-gen2}^{q_{1}} \sum_{i=1}^{m} \GP\sigma_{i}
	+\GP\sigma_0 +H
  \quad \text{in } \dom,
\end{equation}
where
$\conlabel{c:lb-gen2} := (1-q_{1})^{\frac{1}{1-q_{1}}} b^{\frac{-q_1}{(1-q_1)^{2}}}$
and
\[
	\conlabel{c:ub-gen2}
  :=
  	\max\left\{1,
	\left(
		\sum_{i=0}^{m} \|\GP\sigma_{i}\|_{\infty} + \|H\|_{\infty} 
	\right)^{\frac{1}{1-q_{1}}}
	\right\}.
\]
In particular, $\| u \|_{\infty} \le \conref{c:ub-gen2}$.

\item\label{item:2.6-d}
If, in addition $G$ is symmetric,
then any nontrivial supersolution $u$ to \eqref{inteq_gen}
enjoys the lower pointwise estimate
\[
 	u
  \geq
  	\frac{\conref{c:lb-gen2}}{4}
	\left(
	\sum_{i=1}^{m}
	\left[
	(\GP  \sigma_i)^{\frac{1}{1-q_i}}
	+ \K \sigma_i + \GP(H^{q_i} d\sigma_{i})
	\right]
	+ \GP \sigma_0
	\right) + H
  \quad \text{in } \dom.
\]
Moreover, this estimate still holds
with the constant $\conref{c:lb-gen2}/4$ replaced by
$\con = \con(q_1 , b, a)$,
when $G$ is only \textup{QS} with quasi-symmetry constant $a$.
\end{enumerate}

\item\label{item:2.6-ii}
Suppose $G$ is \textup{QM} with quasi-metric constant $h$.
Then the following hold: \medskip 
\begin{enumerate}[label=\textup{(\alph*)}] \itemsep=0.4em \setlength{\itemindent}{-0em}\setcounter{enumii}{4}
\item\label{item:2.6-e}
Any subsolution $u$ to \eqref{inteq_gen}
satisfies the upper pointwise estimate
\[
	u
  \leq
  	\conref{c:ub-gennn}
	\left(
	\sum_{i=1}^{m} \left[(\GP  \sigma_i)^{\frac{1}{1-q_i}}
	+ \K \sigma_i + \GP(H^{q_i} d\sigma_{i})  \right] + \GP \sigma_0
	\right) + H
  \quad \text{in } \dom, 
\]
where $\conlabel{c:ub-gennn} := 2(8mh)^{\frac{q_1}{1-q_{1}}}$.

\item\label{item:2.6-f}
If \eqref{inteq_gen} admits a minimal positive bounded solution,
then it is the only positive bounded solution to \eqref{inteq_gen}. 
\end{enumerate}
\end{enumerate}
\end{theorem}


In arguments below, we often use the inequalities
\begin{equation}\label{eq:c_5}
  	\conref{c:lb-gen2}
  \le
  	(1-q_{i})^{\frac{1}{1-q_i}} b^{\frac{-q_i}{(1-q_i)^2}}
  \le
  	(1-q_{i})^{\frac{1}{1-q_i}} b^{\frac{-q_i}{1-q_i}},
  \quad i=1,\dots,m,
\end{equation}
which follow by $0< q_{i} \leq q_{1} < 1\le b$.

\begin{proof}[Proof of Theorem \ref{bounded}: Part \ref{item:2.6-a}]
Define a sequence of functions
$\lbrace u_{j} \rbrace_{j=0}^{\infty}$ on $\dom$ by
\[
	u_0
  =
  	\conref{c:lb-gen2}
	\sum_{i=1}^{m} (\GP\sigma_{i})^{\frac{1}{1-q_{i}}}
	+\GP\sigma_0+ H  
  \quad \text{and} \quad  
	u_{j}
  =
  	\sum_{i=1}^{m} \GP (u_{j-1}^{q_{i}}d\sigma_{i})
	+\GP\sigma_0+ H
  \quad \text{for } j \in\N. 
\]
Then
$u_0 \geq \conref{c:lb-gen2}  (\GP\sigma_{i})^{\frac{1}{1-q_{i}}}>0$
on $\dom$ by $\sigma_{i} \not\equiv 0$ and $G>0$.
Applying Lemma \ref{lem:iterated} with $s= \frac{1}{1-q_{i}}$
and \eqref{eq:c_5} yields that
\begin{equation}\label{eq:u0u1}
\begin{split}
	u_{1} 
  &=
  	\sum_{i=1}^{m} \GP(u_{0}^{q_{i}}d\sigma_{i})+\GP\sigma_0 + H \\
  &\geq
  	\sum_{i=1}^{m} \conref{c:lb-gen2}^{q_{i}}
	\GP \big( (\GP\sigma_{i})^{\frac{q_{i}}{1-q_{i}}}d\sigma_{i} \big)
	+\GP\sigma_0+ H \\ 
  &\geq
  	\sum_{i=1}^{m} \conref{c:lb-gen2}^{q_{i}}
	(1-q_{i})b^ \frac{-q_i}{1-q_{i}}(\GP\sigma_{i})^{\frac{1}{1-q_{i}}}
	+\GP\sigma_0 + H\\
  &\geq
	\conref{c:lb-gen2} \sum_{i=1}^{m} (\GP\sigma_{i})^{\frac{1}{1-q_{i}}}
	+\GP\sigma_0 + H \\
  &=
  	u_{0}.
\end{split}
\end{equation}
Moreover, by assumptions,
\[
	\| u_{0} \|_{\infty}
  \leq
  	\conref{c:lb-gen2}
	\sum_{i=1}^{m} \|\GP\sigma_{i} \|_{\infty}^{\frac{1}{1-q_{i}}}
	+ \| \GP\sigma_0 \|_{\infty} +\| H\|_{\infty} < + \infty,
\]
and thus 
\[
	\| u_{1} \|_{\infty}
   \leq
	\sum_{i=1}^{m} \|u_{0}\|_{\infty}^{q_{i}} \| \GP\sigma_{i} \|_{\infty}
	+\| \GP\sigma_0 \|_{\infty} +\|H \|_{\infty} < + \infty.
\]
Suppose there exists $j \in \N$ such that
$u_{0} \leq u_{1} \leq \cdots \leq u_{j}$ on $\dom$
with $\|u_j\|_{\infty} < + \infty$. Then
\[
	u_{j+1}
  =
  	\sum_{i=1}^{m} \GP(u_{j}^{q_{i}}d\sigma_{i})
	+ \GP\sigma_0 + H
  \geq
  	\sum_{i=1}^{m} \GP(u_{j-1}^{q_{i}}d\sigma_{i})
	+\GP\sigma_0 + H 
  =
  	u_{j},
\]
and again by assumptions,
\[
	\| u_{j+1} \|_{\infty}
   \leq
	\sum_{i=1}^{m} \|u_{j}\|_{\infty}^{q_{i}} \| \GP\sigma_{i} \|_{\infty}
	+ \| \GP\sigma_0 \|_{\infty} +\|H \|_{\infty} < +\infty.
\]
Hence, by induction,
$\lbrace u_{j} \rbrace_{j=0}^{\infty}$
is a nondecreasing sequence of positive bounded functions in $\dom$.
In addition, since
$0<q_{i} \leq q_{1}<1$ and
$0< \|u_{j-1}\|_{\infty} \leq \|u_{j}\|_{\infty} < + \infty$,
we have
\begin{equation}\label{upper sum}
\begin{split}
	\|u_{j}\|_{\infty}^{1-q_{1}}
&\le
 	\sum_{i=1}^{m}  \frac{\| u_{j-1} \|_{\infty}^{q_{i}}}{\|u_{j} \|_{\infty}^{q_{1}}}
 \| \GP\sigma_{i} \|_{\infty} + \frac{\| \GP\sigma_0 \|_{\infty}}{\|u_{j} \|_{\infty}^{q_{1}}}
	+ \frac{\| H \|_{\infty}}{\|u_{j} \|_{\infty}^{q_{1}}} \\
 &\le
	\sum_{i=1}^{m}  \frac{\| u_{j} \|_{\infty}^{q_{i}}}{\|u_{j} \|_{\infty}^{q_{1}}}
 \| \GP\sigma_{i} \|_{\infty} + \frac{\| \GP\sigma_0 \|_{\infty}}{\|u_{j} \|_{\infty}^{q_{1}}}
	+ \frac{\| H\|_{\infty}}{\|u_{j} \|_{\infty}^{q_{1}}} \\ 
 &\le \sum_{i=0}^{m} \|\GP\sigma_{i}\|_{\infty} + \|H\|_{\infty},
\end{split}
\end{equation}
where the last inequality is valid whenever $\|u_{j}\|_{\infty} >1$.
Therefore
\begin{equation}\label{upper sum 2}
	\|u_j\|_{\infty} 
	  \le
	  \max\left\{1,
	\left( \sum_{i=0}^{m} \|\GP\sigma_{i}\|_{\infty} + \| H \|_{\infty} \right)^{\frac{1}{1-q_{1}}}\right\} = \conref{c:ub-gen2}.
\end{equation}
By the monotone convergence theorem,
the pointwise limit $u = \lim\limits_{j \rightarrow \infty} u_{j}$
is a positive bounded solution to \eqref{inteq_gen}.
Minimality of the solution $u$ follows from its construction.
In fact, if $\tilde{u}$ is a nontrivial supersolution to \eqref{inteq_gen},
then $\tilde{u}\ge\GP(\tilde{u}^{q_i}d\sigma_i)$,
and so
$\GP(\tilde{u}^{q_i}d\sigma_i) \ge \conref{c:lb-gen2}(\GP\sigma_i)^{\frac{1}{1-q_i}}$
by Lemma \ref{lem:Gu>GP} and \eqref{eq:c_5}.
Therefore
\[
	\tilde{u}
  \ge
  	\sum_{i=1}^{m} \GP(\tilde{u}^{q_{i}}d\sigma_{i})
	+\GP\sigma_0 + H
  \ge
  	\conref{c:lb-gen2} \sum_{i=1}^{m} (\GP\sigma_{i})^{\frac{1}{1-q_{i}}}
	+\GP\sigma_0 + H
  =
  	u_{0}.
\]
The definition of $u_j$ and induction yield that
$u_j\leq \tilde{u}$ for all $j \in \mathbb{N}$.
Hence,
$u\leq \tilde{u}$ in $\dom$.
\end{proof}

\begin{proof}[Proof of Theorem \ref{bounded}: Part \ref{item:2.6-b}]
Let $u$ be a nontrivial bounded supersolution to \eqref{inteq_gen}.
Then $u \geq \GP\sigma_{0}$.
Also, since $u\ge\GP(u^{q_i}d\sigma_i)$,
we have $u\ge\conref{c:lb-gen2}  (\GP\sigma_{i})^{\frac{1}{1-q_{i}}}$
by Lemma \ref{lem:lowerbound}.
These imply that $\GP \sigma_i$ must be bounded on $\dom$
for all $i=0,1,\dots, m$. 
\end{proof}

\begin{proof}[Proof of Theorem \ref{bounded}: Part \ref{item:2.6-c}]
Let $u$ be any nontrivial bounded solution to \eqref{inteq_gen}.
By a similar argument to \eqref{upper sum}--\eqref{upper sum 2}
replaced $u_j$ by $u$,
we get $\|u\|_{\infty} \leq \conref{c:ub-gen2}$. 
Therefore
\begin{equation*}\label{upper}
	u
  \leq
  	\sum_{i=1}^{m} \|u\|_{\infty}^{q_i} \GP\sigma_{i}
	+ \GP\sigma_{0} + H
  \leq
  	\conref{c:ub-gen2}^{q_{1}} \sum_{i=1}^{m} \GP\sigma_{i}
	+\GP\sigma_0 + H.
\end{equation*}
Also, since $u\ge\GP(u^{q_i}d\sigma_i)$,
we can use Lemma \ref{lem:Gu>GP} and \eqref{eq:c_5} to obtain
\begin{equation*}\label{lower}
	u
  =
  	\sum_{i=1}^{m} \GP(u^{q_i} d\sigma_{i}) + \GP \sigma_0 + H
 \geq
	\conref{c:lb-gen2} \sum_{i=1}^{m}
	(\GP\sigma_{i})^{\frac{1}{1-q_{i}}}+\GP\sigma_0 + H.
\end{equation*}
These yield the desired inequalities \eqref{eq:estimate}.
\end{proof}

\begin{proof}[Proof of Theorem \ref{bounded}: Part \ref{item:2.6-d}]
Let $u$ be any nontrivial supersolution to \eqref{inteq_gen}.
For $i=1, \dots,m$,
we set $v_i := \GP(u^{q_i}d\sigma_{i})$.
Note that $v_i \le v_i + H \leq u$ in $\dom$.
Since $v_i\ge\GP(v_i^{q_i}d\sigma_i)$,
we derive from Lemmas \ref{lem:Gu>GP} and \ref{symmetricWMP}
and \eqref{eq:c_5} that
\[
	\GP(v_{i}^{q_i}d\sigma_{i})
  \geq
  	\frac{\conref{c:lb-gen2}}{2}
	\left( (\GP\sigma_{i})^{\frac{1}{1-q_{i}}} + \K \sigma_i \right)
  \qtext{in } \dom.
\]
Therefore, by using the elementary inequality 
\[ 
	(a+b)^q \geq \frac{1}{2}(a^q +b^q )
  \quad \text{for } a,b \geq 0 \text{ and } 0<q<1,
\]
we deduce
\begin{align*}
	\GP(u^{q_i} d\sigma_{i})
  &\ge
  	\GP\big( (v_{i}+H)^{q_{i}} d\sigma_i \big)
  \ge
  	\frac{1}{2}
	\big[
		\GP(v_{i}^{q_i} d\sigma_i )+ \GP( H^{q_{i}} d\sigma_i ) 
	\big]  \\
  &\ge
  	\frac{\conref{c:lb-gen2}}{4}
	\left[
	(\GP\sigma_{i})^{\frac{1}{1-q_{i}}} + \K \sigma_i + \GP( H^{q_{i}} d\sigma_i ) 
	\right].
\end{align*}
Therefore
\[
\begin{split}
	u
  &\ge
  	\sum_{i=1}^{m} \GP(u^{q_i} d\sigma_{i}) + \GP \sigma_0 + H  \\
  &\geq
 	\frac{\conref{c:lb-gen2}}{4}
	\left( \sum\limits_{i=1}^{m}
	\left[(\GP  \sigma_i)^{\frac{1}{1-q_i}}
	+ \K \sigma_i + \GP(H^{q_i} d\sigma_{i}) \right]
	+ \GP \sigma_0  \right) +H.
\end{split}
\]
As discussed in Remark \ref{rem_symmetrized},
this estimate remains valid
even when $G$ is QS,
by changing the constant $\conref{c:lb-gen2}/4$.
\end{proof}

\begin{proof}[Proof of Theorem \ref{bounded}: Part \ref{item:2.6-e}]
Let $u$ be any subsolution to \eqref{inteq_gen} and let
\[
	d\nu := \sum\limits_{i=1}^{m} u^{q_i} d\sigma_i + d\sigma_0.
\]
Then $u \leq  \GP \nu +H$ in $\dom$, and  
\[
	d\nu
  \leq
  	\sum_{i=1}^{m} (\GP \nu + H)^{q_i} d\sigma_i + d\sigma_0 
  \leq
 	\sum_{i=1}^{m} (\GP \nu)^{q_i} d\sigma_i
	+ \sum_{i=1}^{m} H^{q_i} d\sigma_i  + d\sigma_0
\]
on Borel subsets of $\dom$.
Therefore
\begin{equation}\label{eq:Gv<GPP}
	\GP \nu
  \leq
  	\sum_{i=1}^{m} \GP\big( (\GP \nu)^{q_i} d\sigma_i \big)
	+ \sum_{i=1}^{m} \GP(H^{q_i} d\sigma_i ) +\GP\sigma_0 .
 \end{equation}
Appealing to Lemma \ref{lemupper}, followed by the inequality
\[ 
	(a+b)^p \leq 2^{p-1}(a^p +b^p )
  \quad \text{for } a,b \geq 0 \text{ and } p > 1,
\]
we estimate
\begin{align*}
	\GP\big( (\GP \nu)^{q_i} d\sigma_i \big)
  &\le
  	(2h)^{q_i} (\GP \nu)^{q_i}
	\left[
	\GP \sigma_i + (\K \sigma_i)^{1-q_i}
	\right]  \\
  &\le
  	(2h)^{q_i} (\GP \nu)^{q_i} 2^{q_i}
	\left[
	(\GP \sigma_i )^{\frac{1}{1-q_i }}
	+ \K \sigma_i
	\right]^{1-q_i}  \\
  &=
  	\Big(\frac{1}{2m} \Big)^{q_i}
	(\GP \nu)^{q_i} \cdot (2m)^{q_i}(4h)^{q_i}
	\left[
	(\GP \sigma_i )^{\frac{1}{1-q_i }} + \K \sigma_i
	\right]^{1-q_i}  \\
  &\le
  	\frac{q_i}{2m} \GP \nu
	+ (1-q_i ) (8mh)^{\frac{q_i}{1-q_i}}
	\left[
	(\GP \sigma_i )^{\frac{1}{1-q_i }} + \K \sigma_i
	\right]  \\
  &\le
  	\frac{1}{2m} \GP \nu
	+ (8mh)^{\frac{q_1}{1-q_1}}
	\left[
	(\GP \sigma_i )^{\frac{1}{1-q_i}} + \K \sigma_i
	\right],
\end{align*}
where we used Young's inequality in the third inequality and,
$0<q_i\le q_1<1$ and $h\ge1/2$ in the last inequality.
Substituting this into \eqref{eq:Gv<GPP} and moving $\GP\nu$ to the left side,
we obtain
\[
	\GP \nu
  \leq
  	\conref{c:ub-gennn}
	\left(
	\sum_{i=1}^{m}
	\left[
	(\GP \sigma_i )^{\frac{1}{1-q_i }} + \K \sigma_i
	+ \GP(H^{q_i} d\sigma_i )
	\right]
	+\GP\sigma_0
	\right).
\]
Hence
\[
	u
  \leq
  	\conref{c:ub-gennn}
	\left(
	\sum_{i=1}^{m}
	\left[
	(\GP \sigma_i )^{\frac{1}{1-q_i }} + \K \sigma_i
	+ \GP(H^{q_i} d\sigma_i )
	\right]
	+\GP\sigma_0
	\right)
	+ H,
\]
as desired.
\end{proof}

\begin{proof}[Proof of Theorem \ref{bounded}: Part \ref{item:2.6-f}]
Let $u$ and $\tilde{u}$ be positive bounded solutions to \eqref{inteq_gen}
such that $u$ is minimal.
It suffices to show that $\tilde{u} \leq u$ in $\Omega$.
By parts \ref{item:2.6-d} and \ref{item:2.6-e}, 
we see that
\[
\begin{split}
	\tilde{u}
  &\leq
  	\conref{c:ub-gennn}
	\left(
	\sum_{i=1}^{m}
	\left[
	(\GP \sigma_i )^{\frac{1}{1-q_i }}
	+ \K \sigma_i + \GP(H^{q_i} d\sigma_i )
	\right]
	+\GP\sigma_0
	\right)
	+ H \\	
  &\leq
  	\kappa
	\left\{
	\frac{\conref{c:lb-gen2}}{4}
	\left(
	\sum_{i=1}^{m}
	\left[
	(\GP \sigma_i )^{\frac{1}{1-q_i }}
	+ \K \sigma_i + \GP(H^{q_i} d\sigma_i )
	\right]
	+\GP\sigma_0
	\right)
	+ H
	\right\} \\
  &\leq
  	\kappa u,
\end{split}
\]
where $\kappa =  4\conref{c:ub-gennn}/\conref{c:lb-gen2} \geq 1$.
 Therefore
\[
	\sum_{i=1}^{m} \GP (\tilde{u}^{q_{i}} d\sigma_{i}) 
  \leq
  	\sum_{i=1}^{m} \kappa^{q_{i}} \GP (u^{q_{i}} d\sigma_{i})
  \leq
  	\kappa^{q_{1}} \sum_{i=1}^{m} \GP (u^{q_{i}} d\sigma_{i}),
\]
and thus 
\[
	\tilde{u}
    =
	\sum_{i=1}^{m} \GP (\tilde{u}^{q_{i}} d\sigma_{i})
	+ \GP\sigma_{0} + H
  \leq
	\kappa^{q_{1}}
	\left[
	\sum_{i=1}^{m} \GP (u^{q_{i}} d\sigma_{i}) + \GP\sigma_{0} + H
	\right] 
  =
  	\kappa^{q_{1}} u .
\]
Iterating this argument, we have
\[
	\tilde{u} \leq \kappa^{q_{1}^{j}} u,
	\quad \forall j\in\N.
\]
Since $0<q_{1}<1$,
letting $j \rightarrow \infty$ yields that $\tilde{u} \leq u$ in $\dom$.
This proves the claim.
\end{proof}

\begin{remark}\label{rem}
We can establish
the uniqueness of a positive bounded solution to \eqref{inteq_gen}
without relying on its minimality among supersolutions.
Indeed, let $u$ and $\tilde{u}$
be positive bounded solutions to \eqref{inteq_gen}.
Arguing as in the proof of Theorem \ref{bounded} 
Part \ref{item:2.6-f} above,
we find a constant $ \kappa \geq 1$ such that
\[ 
	\kappa^{-q_{1}^{j}} u \leq \tilde{u} \leq \kappa^{q_{1}^{j}} u,
  \quad \forall j\in\N.
\]
Therefore $\tilde{u} = u$.
We refer the reader to \cite{VE1},
where a similar argument was employed to prove
the uniqueness of positive solutions to \eqref{inteq_gen}
in the case $m=1$.
 \end{remark}
 
The QM assumption
imposed in Theorem \ref{bounded}~\ref{item:2.6-ii}
is quite restrictive,
since
it fails for the Green function $G_{\LL}$ on a ball and,
more generally, on any bounded uniform domain in $\R^n$.
To address these limitations
while preserving the matching bilateral pointwise bounds
(and thus the uniqueness result),
we explore a broader class of the so-called quasi-metrically modifiable kernels.

A kernel $G:\dom\times\dom\to(0,+\infty]$
is called {\em quasi-metrically modifiable} (QMM)
with a positive function $\mathfrak{m} \in\CF(\dom)$
if the kernel $\tilde{G}$ defined by
\begin{equation}\label{modi}
	\tilde{G}(x,y)
  =
  	\frac{G(x,y)}{\mathfrak{m}(x)\mathfrak{m}(y)},
  \quad x,y \in \dom,
\end{equation}
is QM.
We call $\mathfrak{m}$ a {\em modifier} of $G$.

Examples of QMM kernels are presented in \cite{AN02, FNV, HW}.
In particular,
the Green functions associated with the Laplacian and the fractional Laplacian
of order $0<\alpha<n/2$
are QMM for bounded domains $\dom \subset \R^n$
satisfying the boundary Harnack principle,
such as Lipschitz domains,
or more generally NTA (non-tangentially accessible) domains,
or uniform domains.
Also, the second named author \cite{KH4} established
the generalized $3G$-inequality on a uniform cone in $\R^n$ $(n\geq3)$,
which guarantees that
$G_{-\Delta}$ is QMM with
$\mathfrak{m}$ being
the Martin kernel on $\Omega$ with pole at $\infty$.

We now observe that \eqref{inteq_gen}
is equivalent to the following integral equation in terms of
the modified kernel $\tilde{G}$ with modifier $\mathfrak{m}$:
\begin{equation}\label{modified}
	v
  =
  	\sum_{i=1}^{m} \tilde{\GP} (v^{q_i} d\tilde{\sigma}_{i})
	+
	\tilde{\GP} \tilde{\sigma}_0 + \tilde{H}
  \quad \text{in } \dom,
\end{equation}
where
\begin{equation} \label{definedmeasure}
	v
  :=
  	\frac{u}{\mathfrak{m}}, \quad
	d\tilde{\sigma}_i
  :=
  	\mathfrak{m}^{1+q_i} d\sigma_i, \quad
	d \tilde{\sigma}_0
  :=
  	\mathfrak{m}\,d\sigma_0, \quad
	\tilde{H}
  :=
  	\frac{H}{\mathfrak{m}} .
\end{equation}
In the same spirit as above,
the corresponding modified intrinsic potential
$\tilde{\mathbf{K}}\tilde{\sigma}_i$
is defined by
\[
	\tilde{\K} \tilde{\sigma}_i (x)
  =
  	\int_{0}^{\infty}
	\frac{[\tilde{\varkappa}(B_{\tilde{G}}(x,r))]^{\frac{q_i}{1-q_i}}}{r^2} \, dr,
  \quad x \in \dom,
\]
where for $B=B_{\tilde{G}}(x,r)$,
$\tilde{\varkappa}(B)$ is
the least constant satisfying
\[
 	\| \tilde{\GP} \nu \|_{L^{q_i}(\Omega,\,(\tilde{\sigma}_i)_B)}
  \leq
  	\tilde{\varkappa}(B) \| \nu \|,
  \quad \forall\nu \in \RM.
\]

The following theorem provides
matching lower and upper pointwise bounds
for bounded solutions to \eqref{inteq_gen}
in the case of QMM kernels $G$.

\begin{theorem}\label{cor}
Under the hypotheses of Theorem \ref{bounded},
let $\tilde{G}$ be the modified kernel on $\dom \times \dom$
defined by \eqref{modi} with modifier $\mathfrak{m}$.
\begin{enumerate}[label=\textup{(\roman*)}]
\item
If $G$ is \textup{QS} and $\tilde{G}$ satisfies the \textup{WMP},
then there exists a constant $0<\con<1$ such that
any nontrivial supersolution $u$ to \eqref{inteq_gen}
satisfies the lower pointwise estimate
\begin{equation}\label{lower}
	u
  \geq
  	\con\, \mathfrak{m}
	\left( \sum_{i=1}^{m}
	\left[ (\tilde{\GP} \tilde{\sigma}_i)^{\frac{1}{1-q_i}}
	+
	\tilde{\K}\tilde{\sigma}_i
	+
	\tilde{\GP}(\tilde{H}^{q_i} d\tilde{\sigma}_{i}) \right]
	+
	\tilde{\GP} \tilde{\sigma}_0 \right) + H
  \quad \text{in } \dom.
\end{equation}
\item
If $G$ is \textup{QMM} (i.e., $\tilde{G}$ is \textup{QM}),
then there exists a constant $\con>1$ such that
any subsolution $u$ to \eqref{inteq_gen}
satisfies the upper pointwise estimate
\begin{equation}\label{upper}
	u
  \leq
  	c \, \mathfrak{m}
	\left( \sum_{i=1}^{m}
	\left[(\tilde{\GP}  \tilde{\sigma}_i)^{\frac{1}{1-q_i}}
	+
	\tilde{\K}\tilde{\sigma}_i
	+
	\tilde{\GP}(\tilde{H}^{q_i} d\tilde{\sigma}_{i})  \right]
	+
	\tilde{\GP} \tilde{\sigma}_0 \right) + H
  \quad \text{in } \dom.
\end{equation}
 In this case,
 if \eqref{inteq_gen} admits a positive bounded solution,
 then it is the only positive bounded solution to \eqref{inteq_gen}. 
\end{enumerate}
\end{theorem}

\begin{proof}
Let $v$, $\tilde{\sigma_i}$, $\tilde{\sigma_0}$ and $\tilde{H}$
be defined as in \eqref{definedmeasure}.
By Theorem \ref{bounded}~\ref{item:2.6-d} applied to
$v$, $\tilde{\sigma_i}$, $\tilde{\sigma_0}$, $\tilde{G}$ and $\tilde{H}$,
we obtain 
\[
	v
  \geq
  	\con \left( \sum_{i=1}^{m}
	\Big[ (\tilde{\GP} \tilde{\sigma}_i)^{\frac{1}{1-q_i}}
	+
	\tilde{\K} \tilde{\sigma}_i
	+
	\tilde{\GP}( \tilde{H}^{q_i}d\tilde{\sigma}_i)  \Big]
	+
	\tilde{\GP} \tilde{\sigma}_0 \right) + \tilde{H}
  \quad \text{in } \dom.
\]
This yields \eqref{lower} for the supersolution $u$.
By Theorem \ref{bounded}~\ref{item:2.6-e} applied to
$v$, $\tilde{\sigma_i}$, $\tilde{\sigma_0}$, $\tilde{G}$ and $\tilde{H}$,
we obtain 
\[
	v
  \leq
  	\con \left( \sum_{i=1}^{m}
	\Big[ (\tilde{\GP} \tilde{\sigma}_i)^{\frac{1}{1-q_i}}
	+
	\tilde{\K} \tilde{\sigma}_i
	+
	\tilde{\GP} (\tilde{H}^{q_i}d\tilde{\sigma}_i)  \Big]
	+
	\tilde{\GP} \tilde{\sigma}_0 \right) + \tilde{H}
  \quad \text{in } \dom.
\]
Hence, \eqref{upper} holds for the subsolution $u$.
Moreover, if \eqref{inteq_gen} admits a positive bounded solution,
then it must be unique in view of  Remark \ref{rem}.
\end{proof} 

Let us now prove Theorem \ref{main:thm} and Theorem \ref{fractional},
where $\dom$ is assumed to be regular due to the boundary condition.

\begin{proof}[Proof of Theorem \ref{main:thm}]
Applying Theorem \ref{bounded}~\ref{item:2.6-a}, \ref{item:2.6-b} with
$G= G_{\mathcal L}$
and
$H = H_{f}$,
we see that
\eqref{cond:greenbdd} is valid if and only if
there exists a (minimal) positive bounded solution to \eqref{main:eq}.
Moreover, Theorem \ref{bounded}~\ref{item:2.6-c} implies \eqref{eq:est i}.
The uniqueness follows from Theorem \ref{cor}
when $\dom$ is a bounded uniform domain or
a uniform cone.
 \end{proof}

\begin{proof}[Proof of Theorem \ref{fractional}]
This follows from Theorem \ref{bounded}
with
$\dom = \R^{n}$,
$G(x,y) = k_{2\alpha}(x,y) = |x-y|^{2\alpha - n}$
and $H\equiv 0$.
\end{proof} 

We conclude this section with the following observation.

\begin{remark}
In the case $m = \infty$ (infinitely many sublinear terms),
the analysis of \eqref{inteq_gen} becomes more delicate.
Nevertheless,
our arguments yield that a sufficient condition for
the existence of a (minimal) positive bounded solution to \eqref{inteq_gen} is
\[
	\sum_{i=0}^{\infty} \| \GP \sigma_i \|_{\infty} < +\infty.
\]
Moreover, such a solution satisfies the pointwise bounds
\[
	\conref{c:lb-gen2}
	\sum_{i=1}^{\infty} (\GP\sigma_{i})^{\frac{1}{1-q_{i}}}+\GP\sigma_0 + H
  \le
	u
  \le
	\conref{c:ub-gen3inf}^{q_1}
	\sum_{i=1}^{\infty} \GP\sigma_{i} + \GP\sigma_0 + H
  \quad \text{in } \dom.
\]
Here, $\lbrace q_{i} \rbrace_{i=1}^{\infty} \subset (0,1)$ is nonincreasing
and
\[
	\conlabel{c:ub-gen3inf}
  :=
  	\max\left\{1,
	\left(
	\sum_{i=0}^{\infty} \|\GP\sigma_{i}\|_{\infty} + \|H\|_{\infty}
	\right)^{\frac{1}{1-q_1}}
	\right\}.
\]
Also, analogous results for
\eqref{main:eq} and \eqref{main:frac} with $m = \infty$
can be obtained as well.
\end{remark}
\section{Continuous Solutions to \eqref{main:eq2}}\label{sec:cont}

We begin this section with the following properties of
the $G_{\LL}$-Kato condition (Definition \ref{defkato}),
which can be derived by modifying
the proofs of \cite[Lemmas 2.1--2.3]{HS}.

For $z = \infty$ and $r>0$,
we use the notation
$B(z,r)=(\R^{n} \cup \{ \infty \} ) \setminus \overline{B(0,1/r)}$
and
\[
	\int_{\dom \cap B(z,r)} G_{\LL}(x,y) \, d\sigma (y)
  =
  	\int_{\dom \setminus \overline{B(0,1/r)} } G_{\LL}(x,y) \,d \sigma (y). 
 \]

\begin{lemma} \label{lem:sup}
Let $\sigma \in \RM$.
Then the following are equivalent.
\begin{enumerate}[label=\textup{(\alph*)}]\itemsep=3pt
\item
$\sigma$ satisfies the $G_{\LL}$-Kato condition.
\item
$\dlim_{r \to 0+} \left(\sup_{x \in \dom} \int_{\dom \cap B(z,r)} G_{\LL}(x,y) \, d \sigma (y) \right) = 0$
for any $z \in \cl$.
\item
$\dlim_{r \to 0+} \left(\sup_{z \in \cl} \left( \sup_{x \in \dom} \int_{\dom \cap B(z,r)} G_{\LL}(x,y) \, d \sigma (y) \right) \right) = 0$.
\end{enumerate}
\end{lemma}

\begin{lemma}\label{bdd}
If $\sigma \in \RM$ satisfies the $G_{\LL}$-Kato condition,
then $\LGP \sigma$ is bounded on $\dom$.
\end{lemma}

\begin{lemma}\label{lem:kato}
Let $\sigma \in \RM $ and $C > 0$.
Then the following are equivalent.
\begin{enumerate}[label=\textup{(\alph*)}]\itemsep=3pt
\item
$\sigma$ satisfies the $G_{\LL}$-Kato condition.
\item
$\LGP \sigma_E  \in \CF(\cl)$
for any Borel subset $E$ of $\dom$,
where $\sigma_E$ is the restriction of $\sigma$ to $E$.
\item
$\LGP \sigma \in \CF(\cl)$.
\item
The family
$\{\LGP(\psi \, d \sigma) : \psi \in \Psi \}$
is equicontinuous on $\cl$,
where
\[
	\Psi
  :=
  	\left\{ \psi :0 \leq \psi \leq C
	\text{ $\sigma$-a.e. on $\dom$}
	\right\}.
\]
\end{enumerate}
Moreover, in this case,
$\LGP(\psi \, d \sigma)$
vanishes continuously on $\bdy$
for each $\psi \in \Psi$.
\end{lemma}

We are now ready to prove Theorem \ref{continuous}.

\begin{proof}[Proof of Theorem \ref{continuous}]
For the necessity part, let
$u \in \CF(\cl)$
be a positive solution to \eqref{main:eq2}.
We write
\[
	\LGP \sigma_0
  =
  	u - \sum_{i=1}^{m} \LGP(u^{q_i} d\sigma_{i}) - H_f.
\]
Since the right-hand side is
upper semicontinuous on $\cl$,
we have
$\LGP \sigma_0 \in \CF(\cl)$.
Similarly,
$\LGP(u^{q_i} d\sigma_{i}) \in \CF(\cl)$
for all $i=1,\dots,m$. 
Hence, by  Lemma \ref{lem:kato},
the measures $\sigma_0$ and $u^{q_i} d\sigma_{i}$
satisfy the $G_{\LL}$-Kato condition.
It remains to show that $\sigma_i$ itself satisfies
the $G_{\LL}$-Kato condition for all $i=1, \dots , m$.
To this end, let $z\in \cl$, $r > 0$ and define
$\omega_{i}^{(r)} := (\sigma_i)_{\dom \cap B(z,r)}$.
Then $u\ge\LGP(u^{q_i}d\sigma_i)\ge\LGP(u^{q_i}d\omega_{i}^{(r)})$ in $\dom$.
Appealing to Lemma \ref{lem:Gu>GP} with $b=1$, 
we have
\begin{equation*}\label{est:sup}
	\sup_{\dom} \LGP(u^{q_i} d\omega_{i}^{(r)}) 
  \geq
  	(1-q_{i})^{\frac{1}{1-q_{i}}}
	\sup_{\dom}  (\LGP\omega_{i}^{(r)})^{\frac{1}{1-q_{i}}}.
\end{equation*} 
Since $u^{q_i} d\sigma_i$ satisfies the $G_{\LL}$-Kato condition,
Lemma \ref{lem:sup} gives
\begin{equation*}\label{est:suplim1}
	\lim_{r \to 0^{+}}
	\left(\sup_{x \in \dom} \LGP(u^{q_i}d\omega_{i}^{(r)})(x) \right)=0.
\end{equation*}
Therefore
\[
	\lim_{r \rightarrow 0^{+}}
	\left(\sup_{x \in \dom} \LGP\omega^{(r)}_{i}(x)\right)=0.
\]
Applying Lemma \ref{lem:sup} again, we conclude that 
$\sigma_i$ satisfies the $G_{\LL}$-Kato condition.

For the sufficiency part,
suppose each $\sigma_i$ satisfies the $G_{\LL}$-Kato condition.
According to Lemma \ref{bdd},
$\LGP \sigma_i$ is bounded on $\dom$.
By Theorem \ref{main:thm},
there exists a minimal positive bounded solution $u$ to \eqref{main:eq},
enjoying \eqref{eq:est i}.
Moreover,
$\LGP(u^{q_i} d \sigma_{i}),\LGP \sigma_0\in\CF(\cl)$
by Lemma \ref{lem:kato}.
Thus $u \in \CF(\cl)$ by \eqref{inteq1}.
Hence $u$ is a desired solution to \eqref{main:eq2}.

The uniqueness follows from Theorem \ref{main:thm},
since continuous solutions are bounded on $\dom$.
\end{proof}

We conclude this paper
by presenting an alternative construction of
a continuous solution to \eqref{main:eq2},
employing Schauder's fixed point theorem,
which does not rely on the use of Theorem~\ref{main:thm}.

In the rest, we suppose that
$\sigma_0,\sigma_1,\dots,\sigma_m\in\RM$ satisfy the $G_{\LL}$-Kato condition.
Note that
$\GP_{\mathcal{L}} \sigma_i$ is bounded on $\dom$
by Lemma \ref{bdd}.
Moreover, $H_f \in \CF(\cl)$
since $\dom$ is regular and $f \in \CF^{+}(\bdy)$,
and $\|H_f \|_{\infty} \leq \|f\|_{\infty}$
by the maximum principle.
Define
\[
	v_{0}(x)
  :=
  	\sum_{i=1}^{m} (1-q_{i})^{\frac{1}{1-q_{i}}}
	\LGP\sigma_{i}(x)^{\frac{1}{1-q_{i}}}
	+
	\LGP\sigma_0 (x) + H_f (x),
  \quad x \in \dom,
\]
and
\[
	\conlabel{c:cor}
  :=
  	\max\left\{
	1, \|v_0\|_{\infty},
	\left( \sum_{i=0}^{m} \|\GP_{\mathcal{L}}\sigma_{i}\|_{\infty}
	+ \|f\|_{\infty} \right)^{\frac{1}{1-q_{1}}}
	\right\}.
\]
Consider the class of functions
\[
	\mathcal{P}
  :=
  	\left\{ v \in \CF(\cl) :
	v_0(x) \leq v(x) \leq \conref{c:cor} \text{ for } x \in \dom
	\right\}.
\]
Then $\mathcal{P}$ is a nonempty closed bounded convex subset of $\CF(\cl)$.
Define a linear operator $T_f$ on $\mathcal{P}$ by
\[
	T_f (v)(x)
  := 
  \begin{cases}
	\sum\limits_{i=1}^{m} \LGP(v^{q_i}d \sigma_i)(x)
	+\LGP\sigma_0 (x)+H_{f}(x),
	& x \in \dom,\\[0.8em]
	f(x), & x \in \bdy.
\end{cases}
\]

We shall show that $T_{f}: \mathcal{P} \rightarrow \mathcal{P}$ is a compact operator. 

\begin{lemma}\label{lem:colsed}
$T_f(\mathcal{P}) \subset \mathcal{P}$.
\end{lemma}

\begin{proof}
Let $v \in \mathcal{P}$.
Then
$T_f(v) \in \CF(\cl)$
by Lemma \ref{lem:kato} and $H_f \in \CF(\cl)$.
Since $\conref{c:cor} \geq 1$ and $0<q_i\le q_1$, we have
\[
\begin{split}
	\| T_f(v) \|_\infty
  &\le
  	\sum_{i=1}^m \| v \|_\infty^{q_i} \| \GP_{\mathcal{L}}\sigma_i \|_\infty
	+
	\| \GP_{\mathcal{L}}\sigma_0 \|_\infty + \|f\|_\infty  \\
  &\le
  	\conref{c:cor}^{q_1}
	\left(
	\sum_{i=1}^m  \| \GP_{\mathcal{L}}\sigma_i \|_\infty
	+
	\| \GP_{\mathcal{L}}\sigma_0 \|_\infty + \|f\|_\infty
	\right)  \\
  &\le
  	\conref{c:cor}^{q_1} \cdot \conref{c:cor}^{1-q_1}
  =
  	\conref{c:cor}.
\end{split}
\]
Also, since 
$v \geq v_0 \geq (1-q_{i})^{\frac{1}{1-q_{i}}}(\GP_{\mathcal{L}}\sigma_{i})^{\frac{1}{1-q_{i}}}$ in $\dom$,
it follows from Lemma \ref{lem:iterated} that in $\dom$,
\[
\begin{split}
	T_f(v)
  &\ge
  	\sum_{i=1}^m (1-q_i)^{\frac{q_i}{1-q_i}}
	\LGP\big( (\LGP\sigma_i)^{\frac{q_i}{1-q_i}} d\sigma_i \big)
	+ \LGP\sigma_0 + H_f  \\
  &\ge
  	\sum_{i=1}^m
	(1-q_i)^{\frac{q_i}{1-q_i}} (1-q_i) (\LGP\sigma_i)^{\frac{1}{1-q_i}}
	+ \LGP\sigma_0 + H_f  \\
  &=
  	v_0.
\end{split}
\]
Thus $T_f(v) \in \mathcal{P}$.
\end{proof}

\begin{lemma}\label{lem:com}
$T_f(\mathcal{P}) $ is relatively compact in $\CF(\cl)$.
\end{lemma}

\begin{proof} 
Since $\sigma_{i}$ satisfies the $G_{\mathcal{L}}$-Kato condition
and $H_f \in \CF(\cl)$,
it follows from Lemma~\ref{lem:kato} that
$T_f(\mathcal{P})$ is equicontinuous on $\cl$.
By Lemma~\ref{lem:colsed},
$T_f(\mathcal{P})$ is uniformly bounded on $\cl$. 
Therefore, by Arzel\`a--Ascoli's theorem,
$T_f(\mathcal{P})$ is relatively compact in $\CF(\cl)$.
\end{proof}

\begin{lemma}\label{lem:cont}
$T_f  $ is continuous on $\mathcal{P}$.
\end{lemma}

\begin{proof} 
Let $v_1, v_2 \in \mathcal{P}$ satisfy $\| v_1-v_2 \|_\infty<1$.
Then, in $\dom$,
\begin{align*}
	|T_f(v_1)-T_f(v_2)|
  &\le
  	\sum_{i=1}^m \LGP(|v_1^{q_i}-v_2^{q_i}|d\sigma_i)  \\
  &\le
  	\sum_{i=1}^m \LGP(|v_1-v_2|^{q_i}d\sigma_i)  \\
  &\le
  	\sum_{i=1}^m \| v_1-v_2 \|^{q_i}_\infty \| \LGP\sigma_i \|_\infty.
\end{align*}
Since $0<q_m\le q_i$, we have
\[
	\| T_f(v_1)-T_f(v_2) \|_\infty
  \le
  	\left(
	\sum_{i=1}^m \| \LGP\sigma_i \|_\infty
	\right)
	\| v_1-v_2 \|_\infty^{q_m}.
\]
This implies the continuity of $T_f$ on $\mathcal{P}$.
\end{proof}

\begin{proof}[Proof of Theorem \ref{continuous} (Sufficiency Part)]
The above discussions enable us to employ 
Schauder's fixed point theorem \cite{Ebe}
to establish the existence of 
$u \in \mathcal{P}$ such that $T_f(u) = u$ on $\cl$. 
Clearly,
this $u$ is a positive solution in 
$\CF(\cl)$ to \eqref{main:eq2}. 
\end{proof}
\subsection*{Data Availability Statement}
The authors do not analyze or generate any datasets because this work proceeds with a theoretical and mathematical approach. The relevant materials can be obtained from the references below.

\subsection*{Conflict of Interest} 
On behalf of all authors, the corresponding author states that there is no conflict of interest.

\subsection*{Acknowledgements}
This research is supported by Thailand Science Research and Innovation (TSRI) Fundamental Fund, fiscal year 2026.
Part of this research was conducted during a visit of A.S. to the Centre for Applicable Mathematics, Tata Institute of Fundamental 
Research, Bangalore, whose hospitality is gratefully acknowledged. T.T.S. is supported by the Excellent Foreign Student (EFS) scholarship, Sirindhorn International Institute of Technology (SIIT), Thammasat University. K.H. is supported by JSPS KAKENHI Grant Number JP23K03149. The authors also thank Aye Chan May for helpful discussions on the uniqueness result.
\bibliographystyle{abbrv} 

\end{document}